%% file: complex.tex
\documentclass[graybox]{svmult}

\usepackage{mathptmx}       
\usepackage{helvet}         
\usepackage{courier}        
\usepackage{type1cm}        
\usepackage{amsfonts}
\usepackage{makeidx}         
\usepackage{graphicx}        
\usepackage{multicol}        
\usepackage[bottom]{footmisc}

\makeindex             

\begin{document}
\newtheorem{thm}{Theorem}
\newtheorem{lem}[thm]{Lemma}
\title*{Polynomial estimates over exponential curves in~$\mathbb  C^2$}
\author{Shirali Kadyrov and Yershat Sapazhanov}
\institute{Shirali Kadyrov \at Suleyman Demirel University, Kaskelen, Kazakhstan 040900, \\ \email{shirali.kadyrov@sdu.edu.kz}
\and Yershat Sapazhanov \at Suleyman Demirel University, Kaskelen, Kazakhstan 040900,\\\email{yershat.sapazhanov@sdu.edu.kz}}
%
%
\maketitle

\abstract*{Each chapter should be preceded by an abstract (10--15 lines long) that summarizes the content. The abstract will appear \textit{online} at \url{www.SpringerLink.com} and be available with unrestricted access. This allows unregistered users to read the abstract as a teaser for the complete chapter. As a general rule the abstracts will not appear in the printed version of your book unless it is the style of your particular book or that of the series to which your book belongs.
Please use the 'starred' version of the new Springer \texttt{abstract} command for typesetting the text of the online abstracts (cf. source file of this chapter template \texttt{abstract}) and include them with the source files of your manuscript. Use the plain \texttt{abstract} command if the abstract is also to appear in the printed version of the book.}

\abstract{For any complex $\alpha$ with non-zero imaginary part we show that Bernstein-Walsh type inequality holds on the piece of the curve $\{(e^z,e^{\alpha z}) : z \in \mathbb C\}$. Our result extends a theorem of Coman-Poletsky \cite{CP10} where they considered real-valued $\alpha$.}

\section{Introduction}
\label{sec:1}
In pluripotential theory, one is often interested in growth of polynomials of several variables. A classical Bernstein-Walsh inequality \cite{Ra95} gives important implications in this direction. Recently, there has been significant research work carried in obtaining Bernstein-Walsh type inequalities, see e.g \cite{AO13, Br18, BBL10, CP10, CPo03, CP03, KL16, Ne06}. We now recall the result of Coman-Poletsky in \cite{CP10}.

Let $\alpha \in (0,1) \backslash \mathbb Q$ and $K \subset \mathbb C^2$ be a compact set given by $K=\{(e^z,e^{\alpha z}) : |z| \le 1\}.$ Define
\begin{equation}\label{eq:En}
E_n(\alpha)=\sup\{\|P\|_{\Delta^2} : P \in \mathcal P_n, \|P\|_K \le 1\},
\end{equation}
where $\mathcal P_n$ is the space of polynomials in $\mathbb C[z,w]$ of degree at most $n$, $\Delta^2$ is the closed bidisk $\{(z,w) \in \mathbb C^2 : |z|,|w| \le 1\},$ and $\|\cdot\|_{\Delta^2}, \|\cdot\|_K$ are the uniform norms defined on compact sets $\Delta^2$ and $ K$, respectively. Let $e_n(\alpha):=\log E_n(\alpha)$. Then, Coman-Poletsky prove

\begin{theorem}\label{thm:cp}
For any  Diophantine $\alpha \in (0,1)$ one has
$$\frac{n^2 \log n}{2}-n^2 \le e_n(\alpha) \le \frac{n^2\log n}{2} + 9 n^2 +Cn,$$
for any $n \ge 1$, where constant $C>0$ depends on $\alpha$.
\end{theorem}

Here, term `Diophantine' comes from the Diophantine approximation theory and is an exponent that tells how well a real number can be approximated by rationals. For a proper definition see \cite{CP10}. As a consequence of Theorem~\ref{thm:cp} one gets the Bernstein-Walsh type inequality 
\begin{equation} \label{eq:B}
|P(z,w)| \le \|P\|_K E_n(\alpha) e^{n \log^+\max\{|z|,|w|\}},
\end{equation}
for any $(z,w) \in \mathbb C^2$, $P \in \mathcal P_n$ and $E_n(\alpha)=e^{e_n(\alpha)}$ is determined by the theorem above.

We note that the inequality (\ref{eq:B}) holds for any $\alpha \in \mathbb C$ and finding the optimal bounds for $e_n(\alpha)$ is what makes it challenging in general. The proof of Theorem~\ref{thm:cp} makes use of the well-developed continued fraction expansion theory. We note that the theorem considers real-valued $\alpha$'s only. In this note, we aim to extend Theorem~\ref{thm:cp} to complex $\alpha$'s. 

We now state our main result. 

\begin{theorem} \label{thm:main}
Let $\alpha=\alpha_1 + i \alpha_2$, $\alpha_1,\alpha_2 \in \mathbb R$ be given such that $|\alpha| <1 $ and $\alpha_2 \ne 0$. Then,
$$\frac{n^2 \log n}{2}-n^2  \le e_n(\alpha) \le \frac{n^2 \log n}2+8 n^2-n \log |\alpha_2|.$$
\end{theorem}

We remark here that our proof of Theorem~\ref{thm:main} closely follows that of \cite{CP10}. However, in our case we do not need to appeal to continued fraction theory and as a result our proof requires less effort. Nonetheless, it holds true for \emph{all} non-real complex numbers $\alpha$.

\section{Proof of main result}
\label{sec:2}
In this section we prove our main result Theorem~\ref{thm:main}. For any real $x$ let $\langle x \rangle$ denote the closest integer to $x$. We need the following lemma.

\begin{lem}[cf. Lemma~2.4 in \cite{CP10}]\label{prod}
Let $k,x,y \in \mathbb{Z},$ $x \le y,$ $k \ge 1.$ Then (with $0^0:=1$)
\[ \prod_{j=x}^{y}|j-k \alpha| \ge  \left\{
\begin{array}{ll}
(\frac{y-x}{2e})^{y-x} & \textrm{ , if } \langle k \alpha_1 \rangle \notin [x,y]\textrm{,}\\
(\frac{y-x}{2e})^{y-x} \cdot | k \alpha_2 | & \textrm{ , if } x \le \langle k \alpha_1 \rangle \le y
\end{array} \right. \]
\end{lem}

\begin{proof}

We argue as in Lemma~2.4 of \cite{CP10}. Using Stirling formula 
$$
e^{7/8} \le \frac{m!}{(\frac{m}{e})^m \sqrt{m}} \le e,\forall m \in \mathbb N,
$$
one gets
$$
\prod_{j=1}^{m} (j-\frac{1}{2}) = \frac{(2m)!}{2^{2m} \cdot m!} \ge \left(\frac{m}{e}\right)^m
$$
Let $j_0 =\langle k \alpha_1\rangle$, if $j \ne j_0$ then,
$$
|j-k \alpha| \ge |j- k \alpha_1| = |j-j_0|- \frac{1}{2}
$$
Hence, when $j_0 <x,$ we get
$$
\prod_{j=x}^{y} |j-k\alpha| \ge \frac{1}{2}(y-x)!
$$
Similarly, if $y <j_0,$
$$
\prod_{j=x}^{y} |j-k\alpha| \ge \frac{1}{2}(y-x)!
$$
On the other hand, for $x \le j_0 \le y,$ we obtain
$$
\prod_{j=x}^{y}|j-k \alpha| \ge \Big{(} \frac{j_0-x}{e} \Big{)} ^{j_0-x} \cdot \Big{(} \frac{y - j_0}{e} \Big{)} ^{y - j_0}
\cdot |j_0 - k \alpha| \ge \Big{(}\frac{y-x}{2e} \Big{)}^{y-x} \cdot |j_0-k \alpha|.
$$
However, $|j_0-k \alpha| \ge |k \alpha_2|$, which finishes the proof. \qed
\end{proof}

It is easy to see that the space $\mathcal{P}_n$ of polynomials in $\mathbb C[z,w]$ of degree at most $n$ has dimension equal to $N+1$, where $N := (n^2+3n)/2.$ We are now ready to proceed with the proof of Theorem~\ref{thm:main}.

\begin{proof}
Our argument closely follows that of \cite{CP10}. We first obtain the upper estimate for the exponent $e_n(\alpha)$. For a given polynomial $R(\lambda) = \sum_{j=0}^{m} c_j \lambda^j$ of a single variable we let $D_R$ denote the following differential operator
$$
D_R = R \bigg( \frac{d}{dz} \bigg) = \sum_{j=0}^{m}c_j \frac{d^j}{dz^j }.
$$
Then, $\forall \alpha \in \mathbb{C},$ we have
\begin{equation} \label{diffeq}
D_R  (e^{\alpha z})|_{z=0} = \sum_{j=0}^{m}c_j \alpha^j e^{\alpha \cdot 0} = R(\alpha).
\end{equation}
Let $P(z,w)=\sum_{j+k \le n} c_{jk} z^j w^k \in \mathcal{P}_n$, $ n \ge 1$, be given with $\|P\|_K \le 1$, where as before $K = \{ ( e^z , e^{\alpha z)} : |z| \le 1 \}. $ We set
$$
f(z) := P( e^z , e^{\alpha z}) = \sum_{j+k \le n} c_{jk} e^{(j+k \alpha)z}.
$$
To obtain the upper bound for $e_n(\alpha)$ it suffices to estimate the coefficients $c_{lm}$ from above. To this end, we define the following polynomials $R_{lm}$ of degree $N$
$$
R_{lm}(\lambda) = \prod_{j+k \le n, (j,k) \ne (l,m)} (\lambda - j - k \alpha )=\sum _{t=0}^N a_t \lambda^t.
$$ 
Using (\ref{diffeq}) we have
\begin{eqnarray}
D_{R_{lm}}f(z)|_{z=0} &=& \sum_{j+k \le n}c_{jk}R_{lm}(j+k \alpha) \nonumber\\
                               &=& c_{lm}R_{lm}(l+m \alpha)  \nonumber\\
                               &=& c_{lm} \beta_{lm}  \nonumber \\
            \beta_{lm} &:=& \prod_{j+k \le n, (j,k)\ne (l,m)}(l-j+(m-k)\alpha ) \nonumber
\end{eqnarray}
Using Cauchy's estimates $|f^{(t)}(0)|\le t! \le N^t$ for $t \le N,$ we arrive at
$$
\Big{|} D_{R_{lm}}f(z) |_{z=0} \Big{|} = \Bigg{|} \sum_{t=0}^{N} a_t f^{(t)}(0)\Bigg{|} \le 
\sum_{t=0}^{N} |a_t| N^t \le (N+n)^N,
$$
where the last inequality follows from Vieta's formulas and the fact that \\
$| j+k \alpha | \le n.$
Therefore
\begin{eqnarray} \label{logest}
\log(|c_{lm}\beta_{lm}|) \le N \log (N+n) \le n^2 \log n + 3.7n^2.
\end{eqnarray} 
We now study the lower estimates on $|\beta _{lm}|$ which will lead to upper estimate for $c_{lm}$. Clearly,
$$
|\beta_{lm}| \ge \prod_{k=0,k\ne m}^{n} \prod_{j=0}^{n-k} |l-j+(m-k)\alpha| = A_1 A_2,
$$
Where
$$
A_1 = \prod_{k=0}^{m-1} \prod_{j=0}^{n-k}|j-l-(m-k)\alpha| = \prod_{k=1}^{m} \prod_{j=-l}^{n-l-m+k}|j-k \alpha|,
$$
$$
A_2 = \prod_{k=m+1}^{n} \prod_{j=0}^{n-k}|l-j-(k-m)\alpha| = \prod_{k=1}^{n-m} \prod_{j=l+m+k-n}^{l}|j-k \alpha|,
$$
From Lemma~\ref{prod} we see that
$$
A_1 \ge \prod_{k=1}^{m} \Big{(} \frac{n-m+k}{2e} \Big{)}^{n-m+k} \cdot |k \alpha_2 |,
$$$$
A_2 \ge \prod_{k=1}^{n-m} \Big{(} \frac{n-m+k}{2e} \Big{)}^{n-m+k} \cdot | k \alpha_2|.
$$
Thus, 
\begin{eqnarray}
|\beta _{lm}| \ge A_1 A_2  
		&\ge&  n! \cdot |\alpha_2|^n \cdot \prod_{k=1}^{n} \Big{(} \frac{k}{2e} \Big{)}^{k}\nonumber\\
  		 &\ge&    |\alpha_2|^n \cdot e^{-2n^2} \cdot \prod_{k=1}^n k^k \nonumber.\end{eqnarray}
Using $\prod_{k=1}^n k^k    \ge \frac{n^2 \log n}{2} - \frac{n^2}{4}$ (cf. Lemma~2.1 in \cite{CP03}) we get
\begin{eqnarray}
\log |\beta_{lm}| &\ge& \frac{n^2 \log n}{2} - \frac{n^2}{4}  + n \log|\alpha_2|-2n^2 \nonumber\\
&\ge& \frac{n^2 \log n}{2} - \frac{9n^2}{4} + n \log|\alpha_2|\nonumber
\end{eqnarray}
Thus, using (\ref{logest}) we arrive at
\begin{eqnarray}
\log |c_{lm}| &\le& n^2 \log n + 3.7 n^2 - \log |\beta_{lm}| \nonumber\\
		&\le& \frac{n^2}{2} \log n + 5.95 n^2 - n \log |\alpha_{2}|   \nonumber.
\end{eqnarray}
Clearly, $\|P\|_{\Delta^2} \le \sum_{j+k \le n} |c_{jk}| \le (N+1) \max_{j+k \le n} |c_{jk}|$. Recalling $N = (n^2+3n)/2$ and using (\ref{eq:En}) we obtain that
$$e_n(\alpha)=\log E_n(\alpha) \le \log(N+1) + \frac{n^2}{2} \log n + 5.95 n^2 - n \log |\alpha_{2}|,$$
Since $\log(N+1) \le N = (n^2 +3n)/2 \le 2n^2$ for all $n \ge 1,$ we get that 
$$e_n(\alpha) \le \frac{n^2}{2} \log n + 8 n^2 - n \log |\alpha_{2}|,$$
which gives the upper estimate.

We now turn to obtaining lower estimate for $e_n(\alpha).$ To this end, we would like to construct a polynomial whose $\Delta^2$-norm is large compared to its $K$-norm. We want to  show that we can pick coefficients of $P(z,w)=\sum_{k+j \le n} c_{jk} z^kw^j \in \mathcal P_n$ such that the Maclaurin series expansion of $f(t):=P(e^t, e^{\alpha t})=\sum_{j +k \le n} c_{jk} e^{(k+\alpha j)t}$ is $f(t)=\sum_{k=N}^\infty a_k t^k,$ that is, $f(t)$ has zero of order $N$ at 0, where as before $\dim \mathcal P_n = N+1.$ In other words, we want $f^{(k)}(0)=0$ for all $0 \le k \le N-1.$ Thus, we get a system of $N$ linear equations 
$$\sum_{j+k \le n} c_{jk} (k+\alpha j)^m=0,\,\, m=0,1,\dots,N-1.$$
List $\{(k+\alpha j) : k+j \le n\}$ by $\{a_1,a_2,\dots, a_N\}$, then the system has a solution provided that the following Vandermonde matrix 
\[ \left( \begin{array}{ccccc}
1 & 1 & 1&\dots & 1 \\
a_1 &a_2 & a_3 &\dots & a_N \\
a_1^2 &a_2^2 & a_3^2 &  \dots & a_N^2 \\
\vdots & \vdots & \vdots &\ddots & \vdots\\
a_1^{N-1} & a_2^{N-1} &a_3^{N-1} &\dots & a_N^{N-1} \end{array} \right)\] 
is invertible. Hence, it suffices to show that $a_j \ne a_k$ unless $j=k$. Indeed, $k_1+\alpha j_1=k_2+\alpha j_2$ implies $\alpha_2 j_1=\alpha_2 j_2,$ but $\alpha_2 \ne 0$ so that we must have $j_1=j_2.$ This in turn gives $k_1=k_2.$ So, the system has a solution and we can make sure that $f(t)=P(e^t, e^{\alpha t})=t^N g(t)$ for some entire holomorphic function $g(t).$ We set $h(t):= f(t)/\|P\|_K.$ Then,  $\|h\|_{\Delta } =   \|f\|_\Delta/{\|P\|_K} = 1$ as $\|P\|_K = \sup_{|z|\le 1}|P(e^z,e^{\alpha z})| = \sup_{ |z|\le 1}|f(t)|$. Fix $r \ge 1$ (to be determined later) and consider $|t| = r$, then, Maximum Modulus Principle for holomorphic functions yields
$$\sup_{|t| = r} |h(t)| =\frac{\sup_{|t|= r}|f(t)|}{\|P\|_K} \ge \frac{r^N\sup_{|t| = r} |g(t)|}{\|P\|_K} = 
\frac{r^N  \sup_{|t|= r}|g(t)|}{\|P\|_K} \ge r^N.$$
Equation (\ref{eq:B}) gives
$$|f(t)|=|P(e^t,e^{\alpha t})| \le \|P\|_K E_n(\alpha) e^{n \log^+\max\{|e^t|,|e^{\alpha t}|\}},$$
Hence,
$$r^N \le \sup_{|t| = r} |h(t)| \le E_n(\alpha) e^{n r} \textrm{ for any } r \ge 1.$$
Now taking $r=N/n$ we get
$$e_n(\alpha) \ge N\log(N/n) - N \ge \frac{n^2}2 \log n - n^2.$$
This finishes the proof. \qed
\end{proof}

\begin{acknowledgement}
The first author would like to thank Dan Coman for a useful discussion.
\end{acknowledgement}

\input{referenc}

\end{document}

%% file: referenc.tex
%
%
%

%% file: complex.bbl
\begin{thebibliography}{99.}%
%
%
\bibitem{AO13} Abdullayev, F.G. and Ozkartepe, N.P. (2013).\emph{ An analogue of the Bernstein-Walsh lemma in Jordan regions of the complex plane.} Journal of Inequalities and Applications, (1), p.570.

\bibitem{Br18} Brudnyi, A. (2018). \emph{Bernstein type inequalities for restrictions of polynomials to complex submanifolds of $\mathbb C^N$.} Journal of Approximation Theory, 225, pp.106-147.

\bibitem{BBL10}Bos, L.P., Brudnyi, A. and Levenberg, N. (2010). \emph{On Polynomial Inequalities on Exponential Curves in $\mathbb C^n$.} Constructive Approximation, 31(1), 139--147.

\bibitem{CP10} Coman, D. and Poletsky, E.A. (2010). \emph{Polynomial estimates, exponential curves and Diophantine approximation.} Math. Res. Lett. 17 (6), 1125--1136.

\bibitem{CPo03}Coman, D. and Poletsky, E.A. (2003). \emph{Measures of trancendency for entire functions.} The Michigan Mathematical Journal, 51(3), 575--591.

\bibitem{CP03} Coman, D. and Poletsky, E. (2003). \emph{Bernstein-Walsh inequalities and the exponential curve in $\mathbb C^2$}. Proceedings of the American Mathematical Society, 131(3), 879--887.

\bibitem{KL16} Kadyrov, S. and Lawrence, M. (2016).\emph{ Bernstein-Walsh Inequalities in Higher Dimensions over Exponential Curves.} Constructive Approximation, 44(3), 327--338.

\bibitem{Ne06}Neelon, T., (2006). \emph{A Bernstein-Walsh type inequality and applications}. Canadian Mathematical Bulletin, 49(2), 256--264.

\bibitem{Ra95} Ransford T (1995). \emph{Potential Theory in the Complex Plane,}  Cambridge University Press, Cambridge



%
%
%
%
%
%

\end{thebibliography}
